\documentclass[12pt]{amsart}

\usepackage{amssymb, amscd, txfonts}
\usepackage[all]{xy}
\usepackage{graphicx}

\numberwithin{equation}{section}

\newtheorem{thm}{Theorem}[section]

\newtheorem{lem}[thm]{Lemma}

\newtheorem{conj}[thm]{Conjecture}

\theoremstyle{definition}

\theoremstyle{remark}

\renewcommand{\hom}{\operatorname{Hom}}
\renewcommand{\ker}{\operatorname{Ker}}

\newcommand{\Z}{\mathbb{Z}}

\newcommand{\C}{\mathbb{C}}
\newcommand{\F}{\mathbb{F}}

\DeclareMathOperator{\im}{Im}

\DeclareMathOperator{\tr}{tr}

\begin{document}

\title[Twisted Alexander polynomials and ideal points giving Seifert surfaces]
{Twisted Alexander polynomials and ideal points giving Seifert surfaces}
\author[T.~Kitayama]{Takahiro KITAYAMA}
\address{Department of Mathematics, Tokyo Institute of Technology,
2-12-1 Ookayama, Meguro-ku, Tokyo 152-8551, Japan}
\email{kitayama@math.titech.ac.jp}
\subjclass[2010]{Primary~57M27, Secondary~57Q10}
\keywords{twisted Alexander polynomial, Reidemeister torsion, character variety}

\begin{abstract}
The coefficients of twisted Alexander polynomials of a knot induce
regular functions of the $SL_2(\C)$-character variety. 
We prove that the function of the highest degree has a finite value
at an ideal point which gives a minimal genus Seifert surface
by Culler-Shalen theory.
It implies a partial affirmative answer to a conjecture
by Dunfield, Friedl and Jackson.
\end{abstract}

\maketitle

\section{Introduction}

The aim of this paper is to present an application
of \textit{twisted Alexander polynomials} to \textit{Culler-Shalen theory}
for knots, following a conjecture
by Dunfield, Friedl and Jackson~\cite[Conjecture 8.9]{DFJ}.

In the notable work~\cite{CS} Culler and Shalen established a method
to construct essential surfaces in a $3$-manifold from an ideal point
of the $SL_2(\C)$-character variety.
Their theory applies Bass-Serre theory~\cite{Se1, Se2} to the functional field
of the representation variety.
Twisted Alexander polynomials~\cite{Li, W}, which are known
to be essentially equal to certain Reidemeister torsion~\cite{KL, Kitan},
are invariants of a $3$-manifold associated to linear representations
of the fundamental group.
The torsion invariants generalize many properties of the Alexander polynomial,
and were shown by Friedl and Vidussi~\cite{FV1, FV3} to detect fiberedness
for $3$-manifolds and the Thurston norms of irreducible ones
which are not closed graph manifolds.
We refer the reader to the expositions \cite{Sh} and \cite{FV2}
for literature and related topics on Culler-Shalen theory
and twisted Alexander polynomials respectively.

Let $K$ be a null-homologous knot in a rational homology $3$-sphere.
We denote by $X^{irr}(K)$ the Zariski closure
of the $SL_2(\C)$-character variety of $K$.
Dunfield, Friedl and Jackson~\cite{DFJ} showed
that for each irreducible component $X_0$ in $X^{irr}(K)$
certain normalizations of twisted Alexander polynomials induce
an invariant $\mathcal{T}_K^{X_0} \in \C[X_0][t, t^{-1}]$ called
the \textit{torsion polynomial function} of $K$.
The invariant $\mathcal{T}_K^{X_0}$ satisfies
that $\deg \mathcal{T}_K^{X_0} \leq 4g(K) -2$ and
that $\mathcal{T}_K^{X_0}(\chi)(t^{-1}) = \mathcal{T}_K^{X_0}(\chi)(t)$
for $\chi \in X_0$, where $g(K)$ is the genus of $K$
(cf. \cite[Theorem 1.1]{FK1}, \cite[Theorem 1.5]{FKK}).
For a curve $C$ in $X_0$ we denote by $\mathcal{T}_K^C \in \C[C][t, t^{-1}]$
the restriction of $\mathcal{T}_K^{X_0}$ to $C$,
and by $c(\mathcal{T}_K^C) \in \C[C]$ the coefficient function
in $\mathcal{T}_K^C$ of the highest degree $2g(K)-1$.
It is known that if $K$ is a fibered knot, then $c(\mathcal{T}_K^C)$ is
the constant function with value $1$ (cf. \cite{C, FK1, GKM}).

\begin{conj}[{\cite[Conjecture 8.9]{DFJ}}] \label{conj_DFJ}
If an ideal point $\chi$ of a curve $C$ in $X^{irr}(K)$ gives a Seifert surface
of $K$, then the leading coefficient of $\mathcal{T}_K^C$ has a finite value
at $\chi$.
\end{conj}

In this paper we give a partial affirmative answer
to Conjecture \ref{conj_DFJ}.
The main theorem of this paper is as follows:

\begin{thm} \label{thm_main}
If an ideal point $\chi$ of a curve $C$ in $X^{irr}(K)$ gives
a minimal genus Seifert surface of $K$, then $c(\mathcal{T}_K^C)(\chi)$ is
finite.
\end{thm}

The statement of Theorem \ref{thm_main} is actually weaker than
that of Conjecture \ref{conj_DFJ} on the following two points:
\begin{enumerate}
\item An essential Seifert surface is not necessarily of minimal genus. 
\item If $\deg \mathcal{T}_K^C < 4g(K)-2$, then $c(\mathcal{T}_K^C)(\chi) = 0$
but the leading coefficient of $\mathcal{T}_K^C(\chi)$ is not necessarily finite.
\end{enumerate}
Concerning (1) it should be remarked that classes of knots
with a unique isotopy class of essential Seifert surfaces are known.
For instance, Lyon~\cite[Theorem 2 and Corollary 2.1]{Ly} constructed such a class of non-fibered knots containing $p$-twist knots with $|p| > 1$.

A generalization of Theorem \ref{thm_main} for general $3$-manifolds
will be discussed in a successive work \cite{Kitay}.
See \cite{KKM, KM, Mo} for recent works on other conjectures
by Dunfield-Friedl-Jackson.
 
This paper is organized as follows.
Section $2$ sets up notation and terminology, and provides a brief overview
of Culler-Shalen theory. 
In particular, the precise meaning of `an ideal point giving a surface'
is described.
In Section $3$ we review some basics of Reidemeister torsion, and recalls
properties of torsion polynomial functions.
In this paper we mainly work with Reidemeister torsion
rather than twisted Alexander polynomials, based on the equivalence.
Finally, in Section $4$ we prove Theorem \ref{thm_main}.

\subsection*{Acknowledgment}
This article is prepared for the proceedings of the conference
``The Quantum Topology and Hyperbolic Geometry''
(Nha Trang, Vietnam, May $13$-$17$, $2013$).
The author gratefully acknowledges the organizers' hospitality.
The author would like to thank Stefan Friedl and Takayuki Morifuji
for valuable discussions and helpful comments.
The author also wishes to express his thanks to the anonymous referee
for several useful comments in revising the manuscript.

\section{Culler-Shalen theory}

We begin with briefly reviewing Culler-Shalen theory~\cite{CS, Sh}.
For more details on character varieties we refer the reader to \cite{LM}.

\subsection{Character varieties and ideal points}

Let $M$ be a compact orientable $3$-manifold.
The algebraic group $SL_2(\C)$ acts
on the affine algebraic set $\hom(\pi_1 M, SL_2(\C))$ by conjugation.
The algebro-geometric quotient $X(M)$ of the action is called
the \textit{$SL_2(\C)$-character variety} of $M$.
We denote by $t \colon \hom(\pi_1 M, SL_2(\C)) \to X(M)$ the quotient map.
For a representation $\rho \colon \pi_1 M \to SL_2(\C)$
its \textit{character} $\chi_\rho \colon \pi_1 M \to \C$ is given
by $\chi_\rho(\gamma) = \tr \rho(\gamma)$ for $\gamma \in \pi_1 M$.
The character variety $X(M)$ is known to be realized by the set
of the characters $\chi_\rho$ of $SL_2(\C)$-representations $\rho$,
and $t(\rho) = \chi_\rho$ under the identification.
For $\gamma \in \pi_1 M$ a \textit{trace function}
$I_\gamma \colon X(M) \to \C$ is defined by
$I_\gamma(\chi_\rho) = \tr \rho(\gamma)$ for a representation
$\rho \colon \pi_1 M \to SL_2(\C)$, and it is known
that the coordinate ring of $X(M)$ is generated by
$\{ I_\gamma \}_{\gamma \in \pi_1 M}$.

Let $C$ be a curve in  $X(M)$ which is not necessarily irreducible,
and let $\widehat{C}$ be its smooth projective model.
The points where the rational map $\widehat{C} \to C$ is undefined are called
the \textit{ideal points} of $C$.

Let $K$ be a knot in a rational homology $3$-sphere, and we denote by $E$
its exterior.
In the following we set $X(K) = X(E)$ and denote by $X^{irr}(K)$
the Zariski closure of the subset of $X(K)$ consisting of the characters
of irreducible representations.

\subsection{Essential surfaces given by ideal points}

A non-empty properly embedded compact orientable surface $S$ in $M$ is called
\textit{essential} if for any component $S_0$ of $S$
the homomorphism $\pi_1 S_0 \to \pi_1 M$ induced by the natural inclusion map
is injective, and if no component of $S$ is homeomorphic to $S^2$
or boundary parallel.

Let $\chi$ be an ideal point of a curve $C$ in $X(M)$.
There exists a curve $D$ in $t^{-1}(C)$ such that $t|_D$ is not a constant map,
and that $t|_D$ extends to a regular map $\widehat{D} \to \widehat{C}$
between the smooth projective models.
We take a point $\tilde{\chi}$ of $\widehat{D}$ in the preimage of $\chi$.
Associated to the valuation of $\C(D)$ at $\tilde{\chi}$
Bass-Serre theory~\cite{Se1, Se2} gives a canonical action of $SL_2(\C(D))$
on a tree $T_{\tilde{\chi}}$ without inversions.
Pulling back the action
by the tautological representation $\pi_1 M \to SL_2(\C(D))$, we have an action
of $\pi_1 M$ on $T_{\tilde{\chi}}$. 
Culler and Shalen~\cite[Theorem 2.2.1]{CS} showed that the action is
non-trivial, i.e., for any vertex of $T_{\tilde{\chi}}$ the stabilizer
of the action is not whole the group $\pi_1 M$. 
Now essentially due to Stallings, Epstein and Waldhausen, there exists
a map $f \colon M \to T_{\tilde{\chi}} / \pi_1 M$ such that $f^{-1}(P)$ is
an essential surface, where $P$ is the set of the middle points of edges.
We say that \textit{$\chi$ gives an essential surface $S$} if $S = f^{-1}(P)$
for some $f$ as above.

\section{Torsion invariants}

We review basics of Reidemeister torsion and recall
torsion polynomial functions introduced
by Dunfield, Friedl and Jackson~\cite{DFJ}.
For more details on torsion invariants we refer the reader
to the expositions \cite{Mi, N, T1, T2}.

\subsection{Reidemeister torsion}

Let $C_* = (C_n \xrightarrow{\partial_n} C_{n-1} \to \cdots \to C_0)$ be
a finite dimensional chain complex over a commutative field $\F$,
and let $c = \{ c_i \}$ and $h = \{ h_i \}$ be bases of $C_*$ and $H_*(C_*)$
respectively.
Choose bases $b_i$ of $\im \partial_{i+1}$ for each $i = 0, 1, \dots n$,
and take a basis $b_i h_i b_{i-1}$ of $C_i$ for each $i$ as follows.
Picking a lift of $h_i$ in $\ker \partial_i$ and combining it with $b_i$,
we first obtain a basis $b_i h_i$ of $C_i$.
Then picking a lift of $b_{i-1}$ in $C_i$ and combining it with $b_i h_i$,
we obtain a basis $b_i h_i b_{i-1}$ of $C_i$.
The \textit{algebraic torsion} $\tau(C_*, c, h)$ is defined as:
\[ \tau(C_*, c, h) := \prod_{i=0}^n [b_i h_i b_{i-1} / c_i]^{(-1)^{i+1}} ~\in \F^\times, \]
where $[b_i h_i b_{i-1} / c_i]$ is the determinant of the base change matrix
from $c_i$ to $b_i h_i b_{i-1}$.
If $C_*$ is acyclic, then we write $\tau(C_*, c)$.
It can be easily checked that $\tau(C_*, c, h)$ does not depend on the choice
of $b_i$ and $b_i h_i b_{i-1}$.

Let $(Y, Z)$ be a finite CW-pair.
In the following when we write $C_*(\widetilde{Y}, \widetilde{Z})$,
$\widetilde{Y}$ stands for the universal cover of $Y$
and $\widetilde{Z}$ the pullback of $Z$
by the universal covering map $\widetilde{Y} \to Y$.
For a representation $\rho \colon \pi_1 Y \to GL(V)$
over a commutative field $\F$ we define the twisted homology group as:
\[ H_i^\rho(Y, Z; V) := H_i(C_*(\widetilde{Y}, \widetilde{Z}) \otimes_{\Z[\pi_1 Y]} V). \]
If $Z$ is empty, then we write $H_i^\rho(Y; V)$.

For an $n$-dimensional representation $\rho \colon \pi_1 Y \to GL(V)$
and a basis $h$ of $H_*^\rho(Y, Z; V)$
the \textit{Reidemeister torsion} $\tau_\rho(Y, Z; h)$
associated to $\rho$ and $h$ is defined as follows.
We choose a lift $\tilde{e}$ in $\widetilde{Y}$
for each cell $e \subset Y \setminus Z$.
Then
\[ \tau_\rho(Y, Z; h) := \tau(C_*(\widetilde{Y}, \widetilde{Z}) \otimes_{\Z[\pi_1 Y]} V, \langle \tilde{e} \otimes 1 \rangle_e, h) ~\in \F^\times / (-1)^n \det \rho(\pi_1 Y). \]
If $Z$ is empty or if $H_*^\rho(Y, Z; V) = 0$, then we drop $Z$ or $h$
in the notation $\tau_\rho(Y, Z; h)$.
It can be easily checked that $\tau_\rho(Y, Z; h)$ does not depend
on the choice of $\tilde{e}$ and is invariant under conjugation
of representations.
It is known that Reidemeister torsion is
a simple homotopy invariant.

\subsection{Torsion polynomial functions}

Let $K$ be a null-homologous knot in a rational homology $3$-sphere.
We take an epimorphism $\alpha \colon \pi_1 E \to \langle t \rangle$,
where $\langle t \rangle$ is the infinite cyclic group generated
by the indeterminate $t$. 
For a representation $\rho \colon \pi_1 E \to GL_n(\F)$ we define
a representation $\alpha \otimes \rho \colon \pi_1 E \to GL_n(\F(t))$
by $\alpha \otimes \rho(\gamma) = \alpha(\gamma) \rho(\gamma)$
for $\gamma \in \pi_1 E$.
If  $H_*^{\alpha \otimes \rho}(E; \F(t)^n) = 0$,
then the Reidemeister torsion $\tau_{\alpha \otimes \rho}(E)$ is defined,
and is known by Kirk and Livingston~\cite{KL}, and Kitano~\cite{Kitan}
to be essentially equal to the \textit{twisted Alexander polynomial}
associated to $\alpha$ and $\rho$.
Friedl and Kim~\cite[Theorem 1.1]{FK1} showed that
\[ \deg \tau_{\alpha \otimes \rho}(E) \leq n(2g(K) - 1) \]
(See also \cite{FK2}).
It is known by Cha~\cite{C}, Friedl and Kim~\cite{FK1},
and Goda, Kitano and Morifuji \cite{GKM} that if $K$ is a fibered knot, then
\[ \deg \tau_{\alpha \otimes \rho}(E) = n(2g(K) - 1) \]
and $\tau_{\alpha \otimes \rho}(E)$ is represented by a fraction
of monic polynomials in $\F[t, t^{-1}]$.
See \cite{FV2} for details on twisted Alexander polynomials
and their precise relation with Reidemeister torsion.

Let $X_0$ be an irreducible component of $X^{irr}(K)$.
Dunfield, Friedl and Jackson~\cite[Theorem 1.5]{DFJ} showed that there exists
an invariant $\mathcal{T}_K^{X_0} \in \C[X_0][t, t^{-1}]$ called
the \textit{torsion polynomial function} of $X_0$
such that the following are satisfied for $\chi_\rho \in X_0$:
\begin{itemize}
\item[(i)] If $H_*^{\alpha \otimes \rho}(E; \C(t)^2) = 0$ then,
$\mathcal{T}_K^{X_0}(\chi_\rho) = \tau_{\alpha \otimes \rho}(E) \in \C(t) / \langle t \rangle$.
\item[(ii)] If $H_*^{\alpha \otimes \rho}(E; \C(t)^2) \neq 0$ then,
$\mathcal{T}_K^{X_0}(\chi_\rho) = 0$.
\item[(iii)] $\mathcal{T}_K^{X_0}(\chi_\rho)(t^{-1}) = \mathcal{T}_K^{X_0}(\chi_\rho)(t)$.
\end{itemize}
For a curve $C$ in $X_0$ we denote by $\mathcal{T}_K^C \in \C[X_0][t, t^{-1}]$
the restriction of $\mathcal{T}_K^{X_0}$ to $C$,
and by $c(\mathcal{T}_K^C) \in \C[C]$ the coefficient function
in $\mathcal{T}_K^C$ of the highest degree $2g(K)-1$.

\section{Main theorem}

Now we prove Theorem \ref{thm_main}.
We first prepare key lemmas for the proof.

\subsection{Lemmas}

Let $K$ be a null-homologous knot in a rational homology $3$-sphere and let $S$ be
a minimal genus Seifert surface of $K$.
A tubular neighborhood of $S$ is identified with $S \times [-1, 1]$.
We set $N := E \setminus S \times (-1, 1)$, and denote
by $\iota_\pm \colon S \to N$ the natural homeomorphisms
such that $\iota_\pm(S) = S \times (\pm 1)$.
Since the homomorphisms $\pi_1 S \to \pi_1 E$ and $\pi_1 N \to \pi_1 E$
induced by the natural inclusion maps are injective, in the following
we regard $\pi_1 S$ and $\pi_1 N$ as subgroups of $\pi_1 E$.

\begin{lem} \label{lem_A0}
Let $\rho \colon \pi_1 E \to GL_n(\F)$ be an irreducible representation
with $n > 1$ such that $H_*^{\alpha \otimes \rho}(E; \F(t)^n) = 0$.
Then the following hold:
\begin{itemize}
\item[(i)] $H_0^\rho(S; \F^n) = H_0^\rho(N; \F^n) = H_2^\rho(N; \F^n) = 0$.
\item[(ii)] If $\deg \tau_{\alpha \otimes \rho}(E) = n(2g(K)-1)$,
then $(\iota_\pm)_* \colon H_1^\rho(S; \F^n) \to H_1^\rho(N; \F^n)$ are
isomorphisms.
\end{itemize}
\end{lem}

\begin{proof}
This lemma is proved by techniques developed in \cite{FK1}
together with \cite[Proposition A.3]{FKK}
in terms of twisted Alexander polynomials.
We give only the main steps of the proof with corresponding parts
in the references, and the details are left to the reader.

It follows from \cite[Proposition 3.5]{FK1} and \cite[Proposition A.3]{FKK}
that $H_0^\rho(S; \F^n) = 0$.
Since $H_*^{\alpha \otimes \rho}(E; \F(t)^n) = 0$, the long exact sequence
in \cite[Proposition 3.2]{FK1} implies that
\[  H_i^\rho(N; \F^n) = H_i^\rho(S; \F^n) = 0 \]
for $i = 0, 2$, which proves (i).

It follows from Proof of \cite[Theorem 1.1]{FK1}
that if $\deg \tau_{\alpha \otimes \rho}(E) = n(2g(K)-1)$,
then the inequalities in \cite[Proposition 3.3]{FK1} turn into equalities.
Now (ii) follows from the proof of \cite[Proposition 3.3]{FK1}.
\end{proof}

\begin{lem} \label{lem_A}
Let $\rho \colon \pi_1 E \to GL_n(\F)$ be an irreducible representation
such that $H_*^{\alpha \otimes \rho}(E; \F(t)^n) = 0$.
If $\deg \tau_{\alpha \otimes \rho}(E) = n(2g(K)-1)$, then
\[ \tau_{\alpha \otimes \rho}(E) = \tau_\rho(N, S \times 1) \det(t \cdot id - (\iota_+)_*^{-1} \circ (\iota_-)_*), \]
where $(\iota_\pm)_*$ are
the isomorphisms $H_1^\rho(S; \F^n) \to H_1^\rho(N; \F^n)$.
\end{lem}

\begin{proof}
We pick a basis $h$ of $H_1^{\rho}(S; \F^n)$.
Since
$H_1^{\alpha \otimes \rho}(S; \F(t)^n) = H_1^{\rho}(S; \F^n) \otimes \F(t)$
and
$H_1^{\alpha \otimes \rho}(N; \F(t)^n) = H_1^{\rho}(N; \F^n) \otimes \F(t)$,
$h$ and $(\iota_+)_*(h)$ can be seen also as bases
of $H_1^{\alpha \otimes \rho}(S; \F(t)^n)$
and $H_1^{\alpha \otimes \rho}(N; \F(t)^n)$ respectively.
Taking appropriate triangulations of $E$, $N$ and $S$ and lift of simplices
in the universal covers, we have the following exact sequences: 
\begin{align*}
0 \to C_*(\widetilde{S}) \otimes \F(t)^n \xrightarrow{t (\iota_+)* - (\iota_-)} C_*(\widetilde{N}) \otimes \F(t)^n \to C_*(\widetilde{E}) \otimes \F(t)^n \to 0, \\
0 \to C_*(\widetilde{S}) \otimes \F^n \xrightarrow{(\iota_+)_*} C_*(\widetilde{N}) \otimes \F^n \to C_*(\widetilde{N}, \widetilde{S \times 1}) \otimes \F^n \to 0,
\end{align*}
where the local coefficients in the first and second sequences are understood
to be induced by $\alpha \otimes \rho$ and $\rho$ respectively.
By the multiplicativity of Reidemeister torsion~\cite[Theorem 3.1]{Mi}
and Lemma \ref{lem_A0} we have
\begin{align*}
\tau_{\alpha \otimes \rho}(N; (\iota_+)_*(h)) \det(t \cdot id - (\iota_+)_*^{-1} \circ (\iota_-)_*) &= \tau_{\alpha \otimes \rho}(S; h) \tau_{\alpha \otimes \rho}(E), \\
\tau_{\rho}(N; (\iota_+)_*(h)) &= \tau_{\rho}(S; h) \tau_{\rho}(N, S \times 1).
\end{align*}
By the functoriality of Reidemeister torsion~\cite[Proposition 3.6]{T1} we have
\begin{align*}
\tau_{\alpha \otimes \rho}(N; (\iota_+)_*(h)) &= \tau_{\rho}(N; (\iota_+)_*(h)), \\
\tau_{\alpha \otimes \rho}(S; h) &= \tau_{\rho}(S; h).
\end{align*}
The desired formula now follows from the above equalities.
\end{proof}

\begin{lem} \label{lem_B}
There exists a regular function $f$ of $X(N)$ such that
\[ f(\chi_\rho) = \tau_\rho(N, S \times 1) \]
for a representation $\rho \colon \pi_1 N \to GL_n(\F)$
satisfying that $H_*^\rho(N, S \times 1; \F^n) = 0$.
\end{lem}

\begin{proof}
Let $\rho \colon \pi_1 N \to GL_n(\F)$ be a representation such that $H_*^\rho(N, S \times 1; \F^n) = 0$.
We take a finite $2$-dimensional CW-pair $(V, W)$ with $C_0(V, W) = 0$
which is simple homotopy equivalent to $(N, S \times 1)$.
The differential map
\[ C_2(\widetilde{V}, \widetilde{W}) \otimes_{\F[\pi_1 V]} \F^n \to C_1(\widetilde{V}, \widetilde{W}) \otimes_{\F[\pi_1 V]} \F^n \]
is represented by the matrix $\rho(A)$ obtained as follows from a matrix $A$
in $\Z[\pi_1 V]$.
We first consider the matrix whose $(i, j)$-entries are the image of that of $A$
by $\rho$.
Then we naturally forget the matrix structures of the entries to get
a matrix $\rho(A)$ in $\C$.
By the simple homotopy invariance and the definition of Reidemeister torsion
we have
\[ \tau_\rho(N, S \times 1) = \tau_\rho(V, W) = \det \rho(A). \]
It follows from basics of Linear algebra that $\det \rho(A)$ is written
as a polynomial in $\{ \tr \rho(A)^i \}_{i \in \Z}$,
and that $\tr \rho(A)^i$ is as one
in $\{ \tr \rho(\gamma) \}_{\gamma \in \pi_1 V}$, which proves the lemma.
\end{proof}

The following lemma is a direct corollary of \cite[Theorem 2.2.1]{CS} and \cite[Proposition 2.3.1]{CS}.

\begin{lem} \label{lem_C}
Suppose that an ideal point $\chi$ of a curve in $X^{irr}(K)$ gives
an essential surface $S$.
Then $I_\gamma(\chi) \in \C$ for $\gamma \in \pi_1 E$ represented by a loop
in the complement of $S$.
\end{lem}

\subsection{Proof of the main theorem}

\begin{proof}[Proof of Theorem \ref{thm_main}]
Let $\chi$ be an ideal point of a curve $C$ in $X^{irr}(K)$ which gives
a minimal genus Seifert surface $S$ of $K$,
and let $\rho \colon \pi_1 E \to SL_2(\C)$ be an irreducible representation
such that $\chi_\rho \in C$.
If $H_*^{\alpha \otimes \rho}(E; \C(t)^2) = 0$ and
if $\deg \tau_{\alpha \otimes \rho}(E) = 4g(K)-2$,
then by Lemma \ref{lem_A} we have
\[ c(\mathcal{T}_K^C)(\chi_\rho) = \tau_\rho(N, S \times 1), \]
and so it follows from Lemma \ref{lem_B}
that the function $c(\mathcal{T}_K^C)$ is in the subring of $\C[C]$
generated by $I_\gamma$ for $\gamma \in \pi_1 N$.
Since it follows from Lemma \ref{lem_C} that $I_\gamma(\chi) \in \C$
for $\gamma \in \pi_1 N$, we obtain $c(\mathcal{T}_K^C)(\chi) \in \C$,
which completes the proof. 
\end{proof}


\end{document}